\documentclass[12 pt]{article}
\usepackage[pdftex]{graphicx}
\usepackage{amsmath}
\usepackage{mathtools}
\usepackage{amsthm}
\usepackage{amsfonts}
\usepackage{amssymb}
\usepackage[margin=1.0in]{geometry}
\usepackage{tikz}
\usepackage{enumitem}

\newtheorem{theorem}{Theorem}[section]
\newtheorem{cor}[theorem]{Corollary}
\newtheorem{lemma}[theorem]{Lemma}
\newtheorem{prop}[theorem]{Proposition}
\newtheorem{que}[theorem]{Question}

\theoremstyle{definition}
\newtheorem{defin}[theorem]{Definition}
\newtheorem{fact}[theorem]{Fact}

\theoremstyle{remark}
\newtheorem*{rem}{Remark}
\newtheorem*{claim}{Claim}

\newcommand{\flim}[1]{\mathrm{Flim}(#1)}
\newcommand{\age}[1]{\mathrm{Age}(#1)}
\newcommand{\fin}[1]{\mathrm{Fin}(#1)}
\newcommand{\fr}{Fra\"iss\'e }
\renewcommand{\phi}{\varphi}

\newcommand{\emb}[1]{\mathrm{Emb}(#1)}
\newcommand{\aut}[1]{\mathrm{Aut}(#1)}

\newcommand{\ran}[1]{\mathrm{ran}(#1)}

\newcommand{\cupdots}{\cup\cdots\cup}
\newcommand{\hatf}{\,\hat{\rule{-0.5ex}{1.5ex}\smash{f}}}
\newcommand{\hati}[2]{\hat{\imath}_{#1}^{#2}}
\newcommand{\tildei}[2]{\tilde{\imath}_{#1}^{#2}}

\begin{document}
\title{Thick, syndetic, and piecewise syndetic subsets of \fr structures}
\author{Andy Zucker}
\maketitle

\begin{abstract}
We define and undertake a systematic study of thick, syndetic, and piecewise syndetic subsets of a \fr structure. Each of these collections forms a family in the sense of Akin and Glasner [AG], and we define and study ultrafilters on each of these families, paying special attention to ultrafilters on the thick sets. In the process, we generalize many results of Bergelson, Hindman, and McCutcheon [BHM]. We also discuss some abstract questions about families implicit in the work of Akin and Glasner.
\let\thefootnote\relax\footnote{2010 Mathematics Subject Classification. Primary: 37B05; Secondary: 03C15, 03E15, 05D10, 22F50, 54D35, 54D80, 54H20.}
\let\thefootnote\relax\footnote{Key words and phrases. Fra\"iss\'e theory, topological dynamics, ultrafilters, topological semigroups.}
\let\thefootnote\relax\footnote{The author was partially supported by NSF Grant no.\ DGE 1252522.}
\end{abstract}

\section{Introduction}
Let $G$ be a discrete group, and let $\fin{G}$ denote the finite subsets of $G$. Recall the following definitions (see either [BHM] or [HS]).

\begin{defin}\mbox{}
\label{ComboDef}
\begin{itemize}
\item
$T\subseteq G$ is \emph{thick} if for every $E\in \fin{G}$, there is $g\in G$ with $gE\subseteq T$; equivalently, $T$ is thick if the collection $\{Tg^{-1}: g\in G\}$ has the finite intersection property,
\item
$S\subseteq G$ is \emph{syndetic} if $G\setminus S$ is not thick; equivalently, $S$ is syndetic if there is $E\in \fin{G}$ with $\bigcup_{g\in E} Sg^{-1} = G$,
\item
$P\subseteq G$ is \emph{piecewise syndetic} if there is $E\in \fin{G}$ with $\bigcup_{g\in E} Pg^{-1}$ thick.
\end{itemize}
\end{defin}

Note the left-right switch from the definitions found in [BHM] or [HS]; the reasons for this switch will be explained later. There are several equivalent formulations of these definitions. One such formulation uses the space $2^G$, which we identify with the space of subsets of $G$. $G$ acts on $2^G$ by \emph{right shift}, i.e.\ for $\chi_A\in 2^G$ and $g, h\in G$, we have $\chi_A\cdot g(h) = \chi_A(gh)$. We now have the following equivalent characterizations.

\begin{fact}\mbox{}
\label{ShiftDef}
\begin{itemize}
\item
$T\subseteq G$ is thick iff $\chi_G\in \overline{\chi_T\cdot G}$,
\item
$S\subseteq G$ is syndetic iff $\chi_\emptyset \not\in \overline{\chi_S\cdot G}$,
\item
$P\subseteq G$ is piecewise syndetic iff there is syndetic $S\subseteq G$ with $\chi_S\in \overline{\chi_P\cdot G}$.
\end{itemize}
\end{fact}
\vspace{2 mm}

Another formulation uses $\beta G$, the space of ultrafilters on $G$. We view $G$ as a subspace of $\beta G$ by identifying $g\in G$ with the principal ultrafilter containing $g$. To each subset $A\subseteq G$, we can form the clopen set $\overline{A} := \{p\in \beta G: A\in p\}$, and these clopen sets form a basis for the topology of $\beta G$. We endow $\beta G$ with the structure of a \emph{compact left-topological semigroup}. Compact left-topological semigroups contain minimal right ideals which are always closed. In $\beta G$, each minimal right ideal is also a minimal (right) $G$-flow. We now have the following.

\begin{fact}\mbox{}
\label{SemigroupDef}
\begin{itemize}
\item
$T\subseteq G$ is thick iff $\overline{T}\subseteq \beta G$ contains a minimal right ideal,
\item
$S\subseteq G$ is syndetic iff $\overline{S}\cap M\neq \emptyset$ for every minimal right ideal $M\subseteq \beta G$,
\item
$P\subseteq G$ is piecewise syndetic iff $\overline{P}\cap M \neq\emptyset$ for some minimal right ideal $M\subseteq \beta G$.
\end{itemize}
\end{fact}
\vspace{2 mm}

The primary goal of this paper is to generalize these definitions and characterizations to the setting of countable first-order structures and their groups of automorphisms. We will be especially interested in \emph{ultrahomogeneous} structures, structures where every isomorphism between finite substructures extends to an automorphism of the whole structure. If $\mathbf{K}$ is an ultrahomogeneous structure with automorphism group $G$, we endow $G$ with the \emph{pointwise convergence topology}. If $\mathbf{A}\subseteq \mathbf{K}$ is a finite substructure, then the pointwise stabilizer $G_\mathbf{A} := \{g\in G: \forall a\in \mathbf{A} (g\cdot a = a)\}$ is an open subgroup; these open subgroups form a neighborhood base for $G$ at the identity. Now suppose that $f: \mathbf{A}\rightarrow \mathbf{K}$ is an embedding; we can identify $f$ with the (non-empty!) set of $g\in G$ which extend $f$. Notice that $\{g\in G: g|_{\mathbf{A}} = f\}$ is a left coset of $G_\mathbf{A}$, and we can identify $\emb{\mathbf{A}, \mathbf{K}}$ with the set of left cosets of $G_\mathbf{A}$. Our generalized notions of thick, syndetic, and piecewise syndetic will describe subsets of $\emb{\mathbf{A}, \mathbf{K}}$.

There are three difficulties worth previewing now. First, each finite substructure of $\mathbf{K}$ will give rise to different notions of thick, syndetic, and piecewise syndetic; we will need to understand how these different notions interact with each other. Second, note that when $G$ is a countable discrete group, $G$ acts on itself on both the left and the right. The right action is used in Definition \ref{ComboDef}, and the left action, which is needed to define the right shift action on $2^G$, is used in Fact \ref{ShiftDef}. However, when $G = \aut{\mathbf{K}}$ and $\mathbf{A}\subseteq \mathbf{K}$ is a finite substructure, there is a natural left $G$ action on $\emb{\mathbf{A}, \mathbf{K}}$, but no natural notion of right action. Third, when $G = \aut{\mathbf{K}}$, the space of ultrafilters on $G$ is too fine a compactification to use. The correct compactificaton to use is the \emph{Samuel compactification} of $G$, denoted $S(G)$; we will need an explicit construction of $S(G)$ to allow us to state something similar to Fact \ref{SemigroupDef}.

A secondary goal of this paper will be to consider the notions of thick, syndetic, and piecewise syndetic sets as \emph{families}, collections of subsets of a set $X$ closed upwards under inclusion. Families have been considered by Akin and Glasner [AG] and implicitly by Brian [B]. We will need to study the family of thick sets in some detail; along the way, we will define ultrafilters on families and address some general questions about how these ultrafilters interact with maps.

\section{Background}
\subsection{\fr structures}
Recall that $S_\infty$ is the group of all permutations of a countable set $X$. Given a finite $B\subseteq X$, the \emph{pointwise stabilizer} of $B$ is the set $N_B = \{g\in S_\infty: \forall\, b\!\in \!B\,( g(b) = b)\}$. We can endow $S_\infty$ with the pointwise convergence topology, where a basis of open sets at the identity is given by the collection $\{N_B: B\in \mathcal{P}_{fin}(X)\}$. This turns $S_\infty$ into a Polish group. Note that each $N_B$ is a clopen subgroup of $S_\infty$.

Fix now $G$ a closed subgroup of $S_\infty$. A convenient way to describe the $G$-orbits of finite tuples from $X$ is given by the notions of a \fr class and structure. 

\begin{defin}
Let $L$ be a relational language. A \emph{\fr class} $\mathcal{K}$ is a class of $L$-structures with the following four properties.

\begin{enumerate}
\item
$\mathcal{K}$ contains only finite structures, contains structures of arbitrarily large finite cardinality, and is closed under isomorphism.
\item
$\mathcal{K}$ has the \emph{Hereditary Property} (HP): if $\mathbf{B}\in \mathcal{K}$ and $\mathbf{A}\subseteq \mathbf{B}$, then $\mathbf{A}\in \mathcal{K}$.
\item
$\mathcal{K}$ has the \emph{Joint Embedding Property} (JEP): if $\mathbf{A}, \mathbf{B}\in \mathcal{K}$, then there is $\mathcal{C}$ which embeds both $\mathbf{A}$ and $\mathbf{B}$.
\item
$\mathcal{K}$ has the \emph{Amalgamation Property} (AP): if $\mathbf{A}, \mathbf{B}, \mathbf{C}\in \mathcal{K}$ and $f: \mathbf{A}\rightarrow \mathbf{B}$ and $g: \mathbf{A}\rightarrow \mathbf{C}$ are embeddings, there is $\mathbf{D}\in \mathcal{K}$ and embeddings $r: \mathbf{B}\rightarrow \mathbf{D}$ and $s:\mathbf{C}\rightarrow\mathbf{D}$ with $r\circ f = s\circ g$.
\end{enumerate}
\end{defin}

If $\mathbf{D}$ is an infinite $L$-structure, we write $\age{\mathbf{D}}$ for the class of finite $L$-structures which embed into $\mathcal{D}$. The following is the major fact about \fr classes.

\begin{fact}
If $\mathcal{K}$ is a \fr class, there is up to isomorphism a unique countably infinte $L$-structure $\mathbf{K}$ with $\age{\mathbf{K}} = \mathcal{K}$ satisfying one of the following two equivalent conditions.

\begin{enumerate}
\item
$\mathbf{K}$ is \emph{ultrahomogeneous}: if $f: \mathbf{A}\rightarrow \mathbf{B}$ is an isomorphism between finite substructures of $\mathbf{K}$, then there is an automorphims of $\mathbf{K}$ extending $f$.
\item
$\mathbf{K}$ satisfies the \emph{Extension Property}: if $\mathbf{B}\in \mathcal{K}$, $\mathbf{A}\subseteq \mathbf{B}$, and $f: \mathbf{A}\rightarrow \mathbf{K}$ is an embedding, there is an embedding $h: \mathbf{B}\rightarrow \mathbf{K}$ extending $f$.
\end{enumerate}

Conversely, if $\mathbf{K}$ is a countably infinite $L$-structure satisfying 1 or 2, then $\age{\mathbf{K}}$ is a \fr class.
\end{fact}

Given a \fr class $\mathcal{K}$, we write $\flim{\mathcal{K}}$, the \emph{\fr limit} of $\mathcal{K}$, for the unique structure $\mathbf{K}$ as above. We say that $\mathbf{K}$ is a \emph{\fr structure} if $\mathbf{K}\cong \flim{\mathcal{K}}$ for some \fr class. Our interest in \fr structures stems from the following result.

\begin{fact}
For any \fr structure $\mathbf{K}$, $\aut{\mathbf{K}}$ is isomorphic to a closed subgroup of $S_\infty$. Conversely, every closed subgroup of $S_\infty$ is isomorphic to $\aut{\mathbf{K}}$ for some \fr structure $\mathbf{K}$.
\end{fact}

Throughout the rest of the paper, fix a \fr class $\mathcal{K}$ with \fr limit $\mathbf{K}$. Set $G = \aut{\mathbf{K}}$. We also fix an exhaustion $\mathbf{K} = \bigcup_n \mathbf{A}_n$, with each $\mathbf{A}_n\in \mathcal{K}$ and $\mathbf{A}_m\subseteq \mathbf{A}_n$ for $m\leq n$. Write $H_n = \{gG_n: g\in G\}$, where $G_n = G\cap N_{\mathbf{A}_n}$ is the pointwise stabilizer of $\mathbf{A}_n$; we can identify $H_n$ with $\emb{\mathbf{A}_n, \mathbf{K}}$, the set of embeddings of $\mathbf{A}_n$ into $\mathbf{K}$. Note that under this identification, we have $H_n = \bigcup_{N\geq n} \emb{\mathbf{A}_n, \mathbf{A}_N}$. For $g\in G$, we often write $g|_n$ for $gG_n$, and we write $i_n$ for $G_n$. The group $G$ acts on $H_n$ on the left; if $x\in H_n$ and $g\in G$, we have $g\cdot x = g\circ x$. For $m\leq n$, we let $i^n_m\in \emb{\mathbf{A}_m, \mathbf{A}_n}$ be the inclusion embedding.
\vspace{3 mm}

Each $f\in \emb{\mathbf{A}_m, \mathbf{A}_n}$ gives rise to a dual map $\hatf: H_n\rightarrow H_m$ given by $\hatf(x) = x\circ f$. 

\begin{prop}\mbox{}
\label{Amalgamation}
\begin{enumerate}
\item
For $f\in \emb{\mathbf{A}_m, \mathbf{A}_n}$, the dual map $\hatf: H_n\rightarrow H_m$ is surjective.
\item
For every $f\in \emb{\mathbf{A}_m, \mathbf{A}_n}$, there is $N\geq n$ and $h\in \emb{\mathbf{A}_n, \mathbf{A}_N}$ with $h\circ f = i^N_m$.
\end{enumerate}
\end{prop}

\begin{proof}
Item 1 is an immediate consequence of the extension property. For item 2, use ultrahomogeneity to find $g\in G$ with $g\circ f = i_m$. Let $N\geq n$ be large enough so that $\ran{g|_n}\subseteq \mathbf{A}_N$, and set $h = g|_n$.
\end{proof}

\subsection{Topological dynamics and the Samuel compactification}
We now want to investigate the dynamical properties of the group $G$. Though many of the following definitions and facts hold for more general topological groups, we will only consider the case $G = \aut{\mathbf{K}}$.

A \emph{$G$-flow} is a compact Hausdorff space $X$ equipped with a continuous right $G$-action $a: X\times G\rightarrow X$. We suppress the action $a$ by writing $a(x,g) = x\cdot g$. If $X$ is a $G$-flow, a subflow $Y\subseteq X$ is a closed subspace which is $G$-invariant. The $G$-flow $X$ is \emph{minimal} if it contains no proper subflow; equivalently, $X$ is minimal iff every orbit is dense. By Zorn's lemma, every $G$-flow contains a minimal subflow. 

Given two $G$-flows $X$ and $Y$, a \emph{$G$-map} $f$ is a continuous map $f: X\rightarrow Y$ satisfying $f(x\cdot g) = f(x)\cdot g$ for every $g\in G$. A $G$-flow $X$ is \emph{universal} if for every minimal $G$-flow $Y$, there is a $G$-map $f: X\rightarrow Y$. The following is an important fact of topological dynamics.

\begin{fact}
There is up to $G$-flow isomorphism a unique $G$-flow $M(G)$ which is both universal and minimal. We call $M(G)$ the \emph{universal minimal flow} of $M(G)$. Any $G$-map from $M(G)$ to $M(G)$ is an isomorphism.
\end{fact}

One way of exhibiting the universal minimal flow is to consider the closely related notion of an ambit. A $G$-ambit is a pair $(X, x_0)$ with $X$ a $G$-flow and $x_0\in X$ a distinguished point with dense orbit. If $(X, x_0)$ and $(Y, y_0)$ are $G$-ambits, a \emph{map of ambits} $f$ is a $G$-map $f: X\rightarrow Y$ with $f(x_0) = y_0$. Note that there is at most one map of ambits from $(X, x_0)$ to $(Y, y_0)$, and such a map is always surjective.

\begin{fact}
There is up to $G$-ambit isomorphism a unique $G$-ambit $(S(G), 1)$ which admits a map of $G$-ambits onto every other $G$-ambit. We call $(S(G), 1)$ the \emph{greatest $G$-ambit}. The space $S(G)$ is the \emph{Samuel compactification} of $G$.
\end{fact}

Now suppose $M\subseteq S(G)$ is a subflow. If $Y$ is any minimal $G$-flow, pick any $y_0\in Y$ and form the ambit $(Y, y_0)$. Let $f: S(G)\rightarrow Y$ be the unique map of ambits. Then $f|_M$ is a $G$-map. It follows that $M\cong M(G)$.

We now proceed with an explicit construction of $S(G)$. Let $f\in \emb{\mathbf{A}_m, \mathbf{A}_n}$. The dual map $\hatf$ extends to a continuous map $\tilde{f}: \beta H_n\rightarrow \beta H_m$, where for a discrete space $X$, $\beta X$ is the space of ultrafilters on $X$. If $p\in \beta H_n$ and $f\in \emb{\mathbf{A}_m, \mathbf{A}_n}$, we will sometimes write $p\cdot f$ for $\tilde{f}(p)$. Form the inverse limit $\varprojlim \beta H_n$ along the maps $\tildei{m}{n}$. We can identify $G$ with a dense subspace of $\varprojlim \beta H_n$ by associating to each $g\in G$ the sequence of ultrafilters principal on $g|_n$. The space $\varprojlim \beta H_n$ turns out to be the Samuel compactification $S(G)$ (see Corollary 3.3 in [P]).

To see that $S(G)$ is the greatest ambit, we need to exhibit a right $G$-action on $S(G)$. This might seem unnatural at first; after all, the left $G$-action on each $H_n$ extends to a left $G$-action on $\beta H_n$, and the maps $\tildei{m}{n}$ are all $G$-equivariant, giving us a left $G$-action on $S(G)$. Indeed, let $G_d$ be the group $G$ with the discrete topology; then the map $\pi_m: G_d\rightarrow H_m$ given by $\pi_m(g) = gG_m$ extends to a continuous map $\tilde{\pi}_m: \beta G_d\rightarrow H_m$. As $\pi_m = \tildei{m}{n}\circ \pi_n$ for $n\geq m$, we obtain a map $\pi: \beta G_d\rightarrow S(G)$, and the left action just described on $S(G)$ makes $\pi$ $G$-equivariant, where $\beta G_d$ is given its standard left $G$-action. The problem is that the left action is not continuous when $G$ is given its Polish topology. However, $\beta G_d$ also comes with a natural right $G$-action, and there is a unique way to equip $S(G)$ with a right $G$-action to make the map $\pi$ equivariant. This right $G$-action on $S(G)$  is continuous. For $\alpha\in \varprojlim \beta H_n$, $g\in G$, $m\in \mathbb{N}$, and $S\subseteq H_m$, we have
\begin{align*}
S\in \alpha g(m) \Leftrightarrow \{x\in H_n: x\circ g|_m\in S\}\in \alpha(n)
\end{align*}
where $n\geq m$ is large enough so that $\ran{g|_m}\subseteq \mathbf{A}_n$. Notice that if $g|_m = h|_m = f$, then $\alpha g(m) = \alpha h(m) := \alpha\cdot f := \lambda_m^\alpha(f)$. By distinguishing the point $1\in \varprojlim \beta H_n$ with $1(m)$ principal on $i_m$, we endow $S(G)$ with the structure of a $G$-ambit, and $(S(G), 1)$ is the greatest ambit (see Theorem 6.3 in [Z]). 

Using the universal property of the greatest ambit, we can define a left-topological semigroup structure on $S(G)$: Given $\alpha$ and $\gamma$ in $\varprojlim \beta H_n$, $m\in \mathbb{N}$, and $S\subseteq H_m$, we have
\begin{align*}
S\in \alpha\gamma(m) \Leftrightarrow \{f\in H_m: S\in \alpha\cdot f\}\in \gamma(m).
\end{align*}

If $\alpha\in S(G)$ and $S\subseteq H_m$, a useful shorthand is to put $\alpha^{-1}(S) = \{f\in H_m: S\in \alpha\cdot f\}$. Then the semigroup multiplication can be written as $S\in \alpha\gamma(m)$ iff $\alpha^{-1}(S)\in \gamma(m)$. Notice that for fixed $\alpha$, $\alpha\gamma(m)$ depends only on $\gamma(m)$; indeed, if $\alpha\in \varprojlim \beta H_n$, $p\in \beta H_m$, and $S\subseteq H_m$, we have $S\in \alpha\cdot p$ iff $\alpha^{-1}(S)\in p$. In fact, $\alpha\cdot p = \tilde\lambda_m^\alpha(p)$, where the map $\tilde\lambda_m^\alpha$ is the continuous extention of $\lambda_m^\alpha$ to $\beta H_m$. 

By continuity of left-multiplication, subflows of the $G$-flow $\varprojlim \beta H_n$ are exactly the closed right ideals of the semigroup $\varprojlim \beta H_n$. 

Recall that if $S$ is a semigroup, then $u\in S$ is an \emph{idempotent} if $u\cdot u = u$. 

\begin{fact}
In compact left-topological semigroups, minimal right ideals exist, are always compact, and always contain idempotents. If $Y, Y'\subseteq \varprojlim \beta H_n$ are minimal right ideals, then any $G$-map $\phi: Y\rightarrow Y'$ is of the form $\phi(y) = \alpha y$ for some $\alpha\in Y'$, and each such $G$-map is an isomorphism.
\end{fact}

Furthermore, write $Y = \varprojlim Y_n$, $Y' = \varprojlim Y'_n$, with $Y_m$ and $Y'_m$ compact subsets of $\beta H_m$. Given $\alpha\in Y'$ as above, we note that since $\alpha y(m)$ depends only on $y(m)$, the isomorphism $\phi$ gives rise to homeomorphisms $\phi_m: Y_m\rightarrow Y'_m$ satisfying $\tildei{m}{n}\circ \phi_n = \phi_m\circ \tildei{m}{n}$ for $n\geq m$.

Some quick remarks about notation are in order. Given $f\in \emb{\mathbf{A}_m, \mathbf{A}_n}$ and $S\subseteq H_m$, we will often want to consider the set $(\hat{f})^{-1}(S)$; instead of ``dualling twice,'' we instead write $f(S)$. Note that if $x\in \emb{\mathbf{A}_n, \mathbf{A}_N}$, then we have $x(f(S)) = (x\circ f)(S)$, so this notation is justified. We also remark that whenever we write $\varprojlim Y_m$ for $Y_m \subseteq H_m$ closed, we are implicitly assuming that $\tildei{m}{n}(Y_n) = Y_m$.

\section{The family of thick sets}

If $X$ is an infinite set, a \emph{family} $\mathcal{S}$ on $X$, or just $(X, \mathcal{S})$, is a non-trivial collection of subsets of $X$ closed  upwards under inclusion. By non-trivial, we demand that $X\in \mathcal{S}$ and $\emptyset \not\in \mathcal{S}$. An $\mathcal{S}$-filter $\mathcal{F}$ on $X$ is any filter $\mathcal{F}$ on $X$ with $\mathcal{F}\subseteq \mathcal{S}$. An $\mathcal{S}$-ultrafilter is any maximal $\mathcal{S}$-filter. By Zorn, any $\mathcal{S}$-filter can be extended to some $\mathcal{S}$-ultrafilter. Write $\beta(\mathcal{S})$ for the collection of $\mathcal{S}$-ultrafilters. We borrow the term \emph{family} from Akin and Glasner [AG], and the definition of an $\mathcal{S}$-filter is implicit in [B]. 
 
If $(X,\mathcal{S})$ is a family and $f: X\rightarrow Y$ is a map, we can push forward the family $\mathcal{S}$ to the family $f(\mathcal{S}) := \{B\subseteq Y: f^{-1}(B)\in \mathcal{S}\}$. In a similar way, any $\mathcal{S}$-filter $\mathcal{F}$ pushes forward to an $f(\mathcal{S})$-filter $f(\mathcal{F})$. We call a map $f: (X, \mathcal{S})\rightarrow Y$ \emph{strong} if $f^{-1}(f(A))\not\in \mathcal{S}$ whenever $A\not\in\mathcal{S}$; we call $f$ \emph{regular} if $f(\beta(\mathcal{S}))\subseteq \beta(f(\mathcal{S}))$, i.e.\ if the push forward of any $\mathcal{S}$-ultrafilter is an $f(\mathcal{S})$-ultrafilter. Note that we always have $\beta(f(\mathcal{S}))\subseteq f(\beta(\mathcal{S}))$, since given $q\in \beta(f(\mathcal{S}))$, any $p\in \beta(\mathcal{S})$ extending the $\mathcal{S}$-filter generated by $\{f^{-1}(B): B\in q\}$ will satisfy $f(p) = q$. Here is a quick proposition to give some intuition about the definitions.

\begin{prop}
Any strong map is also regular.
\end{prop}

\begin{proof}
\label{StrongReg}
Suppose $f: (X, \mathcal{S})\rightarrow Y$ is strong, and let $p\in \beta(\mathcal{S})$. To show  $f(p)\in \beta(f(\mathcal{S}))$, let $B\subseteq Y$ with $B\in f(\mathcal{S})\setminus f(p)$. Then $f^{-1}(B) \in \mathcal{S}\setminus p$. As $p\in \beta(\mathcal{S})$, there must be $A\in p$ with $A\cap f^{-1}(B)\not\in \mathcal{S}$. As $f$ is strong, we have $f(A\cap f^{-1}(B)) = f(A)\cap B\not\in f(\mathcal{S})$. As $f(A)\in f(p)$, we have shown that $f(p)\in \beta(f(\mathcal{S}))$.
\end{proof}

We will be interested in studying the family $\mathcal{T}_m$ of thick subsets of $H_m$, and in particular understanding the members of $\beta(\mathcal{T}_m)$. Our definitions of thick, syndetic, and piecewise syndetic will most closely follow the second set of definitions from the introduction. Indeed, since $G$ acts on the left on $H_m$, this induces a right shift action on $2^{H_m}$ \newline

\begin{defin}\mbox{}
\label{MainDef}
\begin{itemize}
\item
$T\subseteq H_m$ is \emph{thick} iff $\chi_{H_m}\in \overline{\chi_T\cdot G}$,
\item
$S\subseteq H_m$ is \emph{syndetic} iff $\chi_\emptyset\not\in \overline{\chi_S\cdot G}$,
\item
$P\subseteq H_m$ is \emph{piecewise syndetic} iff there is syndetic $S\subseteq H_m$ with $\chi_S\in \overline{\chi_P\cdot G}$.
\end{itemize}
\end{defin}

Each of these notions is upwards closed, so forms a family. Write $\mathcal{S}_m$ and $\mathcal{P}_m$ for the families of syndetic and piecewise syndetic subsets of $H_m$, respectively. Note that $T\subseteq H_m$ is thick iff $H_m\setminus T$ is not syndetic, and vice versa.

Let us briefly return to the setting of the introduction, where $G$ is a discrete group. Definition \ref{ComboDef} uses the right $G$-action on itself; given $T\subseteq G$, we need the right action to describe the sets $Tg^{-1}$ for $g\in G$. But the definition of ``thick'' given there also includes a more combinatorial characterization, where $T$ is thick iff for every $E\in \fin{G}$, there is $g\in G$ with $gE\subseteq T$. When $G = \mathbb{Z}$ for instance, this amounts to saying that thick subsets of the integers are exactly those subsets containing arbitrarily long intervals. Since this characterization uses the left $G$-action on itself, we can hope to generalize it. We now return to the normal setting of the paper, where $G = \aut{\mathbf{K}}$.

\begin{prop}
\label{Combinatorial}
$T\subseteq H_m$ is thick iff for every $n\geq m$, there is $x\in H_n$ with \newline $x\circ \emb{\mathbf{A}_m, \mathbf{A}_n}\subseteq T$.
\end{prop}

\proof
Suppose $T\subseteq H_m$ is thick, and let $g_n\in G$ be a sequence of group elements with $\chi_T\cdot g_n\rightarrow \chi_{H_m}$. By passing to a subsequence, we may assume that for each $n\geq m$, we have $\chi_T\cdot g_n(f) = 1$ for every $f\in \emb{\mathbf{A}_m, \mathbf{A}_n}$. It follows that $g_n|_n\circ \emb{\mathbf{A}_m, \mathbf{A}_n}\subseteq T$.
\vspace{3 mm}

Conversely, suppose $T\subseteq H_m$ is a set so that for every $n\geq m$, there is $x_n\in H_n$ with $x_n\circ \emb{\mathbf{A}_m, \mathbf{A}_n}\subseteq T$. Find $g_n\in G$ with $g_n|_n = x_n$. Then $\chi_\cdot g_n\rightarrow \chi_{H_m}$, and $T$ is thick.
\qedhere
\vspace{3 mm}

Call a sequence $\vec{X} = \{x_n: n< \omega\}$ with each $x_n\in H_n$ an \emph{exhausting sequence}. Writing $\emb{\mathbf{A}_m, \vec{X}} := \bigcup_{n\geq m} (x_n\circ \emb{\mathbf{A}_m, \mathbf{A}_n})$, we can say that $T\subseteq H_m$ is thick iff there is an exhausting sequence with $\emb{\mathbf{A}_m, \vec{X}}\subseteq T$.

We now want something akin to the third set of definitions from the introduction. Before we can attempt a generalization, we need to understand how the thick families on different levels interact. The following proposition is implicit in section 4 of [Z] (see Lemma 4.2).

\begin{prop}
\label{ThickMap}
For any $f\in \emb{\mathbf{A}_m, \mathbf{A}_n}$, we have that $\hat{f}(\mathcal{T}_n) = \mathcal{T}_m$. 
\end{prop}

\begin{proof}
We need to show that $T\subseteq H_m$ is thick iff $f(T)\subseteq H_n$ is thick. Assume first that $T$ is thick and $\vec{X}$ is an exhausting sequence with $\emb{\mathbf{A}_m, \vec{X}}\subseteq T$. Then $\emb{\mathbf{A}_n, \vec{X}}\subseteq f(T)$ (since $\emb{\mathbf{A}_n, \vec{X}}\circ f\subseteq \emb{\mathbf{A}_m, \vec{X}}$). 

For the other direction, suppose $f(T)$ is thick and $\vec{X}$ is an exhausting sequence with $\emb{\mathbf{A}_n, \vec{X}}\subseteq f(T)$. For each $N\geq m$, find $\ell_N \geq N$ large enough so that for each $h\in \emb{\mathbf{A}_m, \mathbf{A}_N}$, there is $x_h\in \emb{\mathbf{A}_n, \mathbf{A}_{\ell_N}}$ with $x_h\circ f = h$. Let $\vec{Y}$ be the exhausing sequence given by $y_N = x_{\ell_N}\circ i^{\ell_N}_N$. Then $\emb{\mathbf{A}_m, \vec{Y}}\subseteq T$.
\end{proof}

A consequence of Proposition \ref{ThickMap} is that for $n\geq m$, $f\in \emb{\mathbf{A}_m, \mathbf{A}_n}$, and $F$ a $\mathcal{T}_n$-filter on $H_n$, we have that $\hatf(F)$ is a $\mathcal{T}_m$-filter on $H_m$. This will be especially important when $f = i_m^n$.

If $T\subseteq H_m$, set $\overline{T} = \{p\in \beta H_m: T\in p\}$, and set $\tilde{T} = \{\alpha\in \varprojlim \beta H_n: T\in \alpha(m)\}$. If $Y_m\subseteq \beta H_m$ is closed, then the \emph{filter of clopen neighborhoods} of $Y_m$ is the filter $F_m$ on $H_m$ so that $Y_m = \bigcap_{T\in F_m} \overline{T}$. Conversely, if $F_m$ is a filter on $H_m$, then the \emph{closed set} for $F_m$ is the set $\bigcap_{T\in F_m} \overline{T}$. The following two theorems are the key to generalizing Fact \ref{SemigroupDef}

\begin{theorem}
\label{ThickUlts}
Let $Y = \varprojlim Y_n\subseteq \varprojlim \beta H_n = S(G)$ be compact, with $Y_m\subseteq \beta H_m$ for $m\in \mathbb{N}$. Let $F_m$ be the filter of clopen neighborhoods of $Y_m$. Then $Y$ is a minimal right ideal iff each $F_m$ is a member of $\beta(\mathcal{T}_m)$.
\end{theorem}

\begin{theorem}
\label{RegularMaps}
If $f\in \emb{\mathbf{A}_m, \mathbf{A}_n}$, then $\hatf: (H_n, \mathcal{T}_n)\rightarrow H_m$ is regular,
\end{theorem}

As a consequence of both Theorems \ref{ThickUlts} and \ref{RegularMaps}, notice that if $F_m\in \beta(\mathcal{T}_m)$, then there is a minimal right ideal $Y = \varprojlim Y_m\subseteq S(G)$ with $F_m$ the filter of clopen neighborhoods of $Y_m$. In particular, since every thick $T\subseteq H_m$ belongs to some $F\in \beta(\mathcal{T}_m)$, we have the following corollary, a generalization of Theorem 2.9(d) from [BHM].

\begin{cor}
$T\subseteq H_m$ is thick iff $\tilde{T}\subseteq S(G)$ contains a minimal right ideal.
\end{cor}

The first step is to characterize right ideals of $\varprojlim \beta H_n$. The following is a generalization of Theorem 8.7 from [Z]. 

\begin{prop}
\label{ContainRIdeal}
Let $Y = \varprojlim Y_n\subseteq \varprojlim \beta H_n$ be compact, with $Y_m\subseteq \beta H_m$ for $m\in \mathbb{N}$. Let $F_m$ be the filter of clopen neighborhoods of $Y_m$. Then $Y$ contains a right ideal iff each $F_m$ is a $\mathcal{T}_m$-filter.
\end{prop}

\begin{proof}  
First assume that each $Y_m$ is thick. For $W\subseteq G$ finite, $m\in\mathbb{N}$, and $S\in F_m$, let $Y_{W, S}$ consist of those $\alpha\in Y$ such that $S\in \alpha g(m)$ for each $g\in W$. Notice that $Y_{W,S}\subseteq Y$ is closed, hence compact.
\vspace{2 mm} 

\begin{claim} 
First, let us show that $Y_{W, S}$ is nonempty. Fix $n$ large enough so that $g(\mathbf{A}_m)\subseteq \mathbf{A}_n$ for each $g\in W\cup \{1_G\}$. For $g\in G$, set $T_g = \{f\in H_n: f\circ g|_m\in S\}$. We will show that the set $X := T_{1_G}\setminus \left(\bigcap_{g\in W} T_g\right)$ is not thick. If it were, pick $N$ large enough so that $g(\mathbf{A}_n)\subseteq \mathbf{A}_N$ for each $g\in W\cup \{1_G\}$ and find $h\in H_N$ so that $h\circ \emb{\mathbf{A}_n,\mathbf{A}_N}\subseteq X$. But now for each $g\in W$, set $x_g = h\circ g|_n\circ i_m^n = h\circ i^N_n\circ g|_m$. Since $g|_n$ and $i^N_n$ are both in $\emb{\mathbf{A}_n,\mathbf{A}_N}$, we have $h\circ g|_n$ and $h\circ i^N_n$ in $X$. Since $h\circ g|_n\in X\subseteq T_{1_G}$, we have $x_g\in S$. But this implies that $h\circ i^N_n\in\bigcap_{g\in W} T_g$, a contradiction.

Since $F_n\in \mathcal{T}_n$ and since $T_{1_G} = i_m^n(S)\in Y_n$, this means that $(\bigcap_{g\in W\cup \{1_G\}} T_g) \in p$ for some $p\in Y_n$. Now any $\alpha\in Y$ with $\alpha (n) = p$ is a member of $Y_{W,S}$. This proves the claim.
\end{claim}
\vspace{2 mm}

Now observe that if $W_1,W_2$ are finite subsets of $G$, $S_1\in F_m$, and $S_2\in F_n$ ($m\leq n$), then letting $S_3 = i_m^n(S_1)\cap S_2\in F_n$, we have $Y_{W_1,S_1}\cap Y_{W_2, S_2} \supseteq Y_{W_1\cup W_2, S_3}$. In particular, since each $Y_{W,S}$ is compact, there is $\alpha\in Y$ a member of all of them. Hence $\overline{\alpha\cdot G}\subseteq Y$ is a subflow of $\varprojlim \beta H_n$. 
\vspace{3 mm}

For the other direction, suppose there were $S\in F_m$ with $S\not\in \mathcal{T}_m$. Pick $\alpha\in Y$; we need to show that for some $g\in G$, $S\not\in \alpha g(m)$. To the contrary, suppose that $S\in \alpha g(m)$ for every $g\in G$, or equivalently, that $S\in \alpha\cdot f$ for every $f\in H_m$. But then $\{x\in H_n: x\circ f\in S \textrm{ for every } f\in \emb{\mathbf{A}_m, \mathbf{A}_n}\}\in \alpha(n)$, so in particular is non-empty. But this implies that $S$ is thick, a contradiction.
\end{proof}

Before proving Theorem \ref{ThickUlts}, we need a quick remark about topological dynamics. If $X$ is a $G$-flow, then each point $x\in X$ gives rise to the ambit $(\overline{x\cdot G}, x)$. As such, there is a map of ambits $\phi: (S(G), 1)\rightarrow (\overline{x\cdot G}, x)$, and we write $\phi(\alpha) := x\cdot \alpha$. Since $\phi$ is surjective, we have $x\cdot S(G) = \overline{x\cdot G}$. If $\alpha, \gamma\in S(G)$, we also have $x\cdot (\alpha\gamma) = (x\cdot \alpha)\cdot \gamma$. Now suppose that $X = 2^{H_m}$, and $T\subseteq H_m$. Then if $\alpha\in S(G)$, we have $\chi_T\cdot \alpha = \chi_{\alpha^{-1}(T)}$. 

\begin{proof}[Proof of Theorem \ref{ThickUlts}]\mbox{}\\*
$(\Leftarrow)$ We prove a slightly stronger statement. Suppose $Y = \varprojlim Y_n$ is a closed subset of $\varprojlim \beta H_n$ with $F_m$ the filter of clopen neighborhoods of $Y_m$. Say that for infinitely many $m\in \mathbb{N}$, we have $F_m\in \beta(\mathcal{T}_m)$. By Proposition \ref{ThickMap}, we see that for every $m\in \mathbb{N}$, $F_m$ is a $\mathcal{T}_m$-filter, so by Proposition \ref{ContainRIdeal}, $Y$ contains some right ideal. If $Z = \varprojlim Z_n \subsetneq Y$ is any proper closed subset, then $Z_m\subsetneq Y_m$ for any large enough $m\in \mathbb{N}$. So fix a large $m\in \mathbb{N}$ with $F_m\in \beta(\mathcal{T}_m)$. The filter $G_m$ of clopen neighborhoods of $Z_m$ then properly contains $F_m$ and thus cannot be in $\mathcal{T}_m$. Hence $Z$ is not a right ideal, so $Y$ is a minimal right ideal.

$(\Rightarrow)$ If $\alpha\in \varprojlim \beta H_n$, then the left-multiplication map $p\rightarrow \alpha\cdot p$ on $\beta H_m$ is continuous. Hence if $X\subseteq \beta H_m$ is a closed subset defined by the filter $F$, $\alpha X\subseteq \beta H_m$ is also closed. It is defined by the filter
\begin{align*}
\alpha F := \{S\subseteq H_m: \alpha^{-1}(S)\in F\}.
\end{align*}
Now suppose that $F$ is a $\mathcal{T}_m$-filter. Then for any $\alpha\in \varprojlim \beta H_n$, $\alpha F$ is also a $\mathcal{T}_m$-filter. This is because $\alpha^{-1}(S)\in \mathcal{T}_m$ implies that for some $\alpha'\in \varprojlim \beta H_n$, we have $\chi_S\cdot \alpha\alpha' = \chi_{H_m}$. But this implies that $S\in \mathcal{T}_m$.

Let $Y = \varprojlim Y_m$ be a minimal right ideal, with $F_m$ the filter of clopen neighborhoods of $Y_m$. Suppose for sake of contradiction that for some $m\in \mathbb{N}$, $F_m\not\in \beta(\mathcal{T}_m)$. Say $G_m\supsetneq F_m$ with $G_m\in \beta(\mathcal{T}_m)$. Inductively define for each $n\geq m$ $G_n\in \beta(\mathcal{T}_n)$ so that $G_{n+1}$ is any $\mathcal{T}_{n+1}$-ultrafilter with $\hati{n}{n+1}(G_{n+1}) = G_n$. If $Z_m\subseteq \beta H_m$ is the closed set for $G_m$, then by the proof of the right to left implication, $\varprojlim Z_n$ is a minimal right ideal. Fix $\alpha\in Z$, and consider $\alpha G_m$. As the left-multiplication map $p\rightarrow \alpha\cdot p$ is a homeomorphism from $Y_m$ to $Z_m$, we have $\alpha(Z_m)\subsetneq Z_m$, so $\alpha G_m\supsetneq G_m$. However, $\alpha G_m$ is a $\mathcal{T}_m$-filter, and $G_m$ is maximal. Taken together, this is a contradiction.
\end{proof}
\vspace{3 mm}

We now turn to the proof of Theorem \ref{RegularMaps}. First an easy lemma.

\begin{lemma}
\label{RegMap}
If $f: (X,\mathcal{S})\rightarrow Z$ is a regular map, and $f = g\circ h$, with $h: X\rightarrow Y$ and $g: Y\rightarrow Z$, then $g: (Y, h(\mathcal{S}))\rightarrow Z$ is regular.
\end{lemma}

\begin{proof}
Let $q\in \beta(h(\mathcal{S}))$, and set $r = g(q)$. Pick $p\in \beta(\mathcal{S})$ with $h(p) = q$. Then $f(p) = r$, so $r\in \beta(f(\mathcal{S})) = \beta (g(h(\mathcal{S})))$.
\end{proof}

\begin{proof}[Proof of Theorem \ref{RegularMaps}]\mbox{}\\*
 By Proposition \ref{Amalgamation} and Lemma \ref{RegMap}, it is enough to show that each $\hati{m}{n}$ is regular. Say $F_n\in \beta(\mathcal{T}_n)$. Inductively define $F_N\in \beta(\mathcal{T}_N)$ for each $N\geq n$ so that $\hati{N+1}{N}(F_{N+1}) = F_N$; letting $Y_n\subseteq \beta H_n$ be the closed set for $F_n$, we have seen that $\varprojlim Y_n$ is a minimal right ideal. It follows that $F_m = \hati{m}{n}(F_n)$ is in $\beta(\mathcal{T}_m)$.
\end{proof}
\vspace{3 mm}

We end this section with a brief ``sanity check.'' Recall from section 2 the map $\pi_m: G_d\rightarrow H_m$ with $\pi_m(g) = g|_m$.

\begin{theorem}
\label{Sanity}
$T\subseteq H_m$ is thick iff $\pi_m^{-1}(T)\subseteq G_d$ is thick.
\end{theorem}

\proof
Say $T\subseteq H_m$ is thick. Let $g_1,...,g_k\in G_d$. Write $f_i = \pi_m(g_i)$, and find $n\geq m$ so that for each $i\leq k$, we have $f_i\in \emb{\mathbf{A}_m, \mathbf{A}_n}$. By Proposition \ref{Combinatorial}, there is $x\in H_n$ with $x\circ \emb{\mathbf{A}_m, \mathbf{A}_n}\subseteq T$. Find $g\in G_d$ with $g|_n = x$. Then for each $i\leq k$, we have $(gg_i)|_m\in T$, so $g\in \bigcap_{i\leq k} \pi_m^{-1}(T)g_i^{-1}\neq \emptyset$, and $\pi_m^{-1}(T)$ is thick.

Conversely, say $T$ is not thick, so that $S := H_m\setminus T$ is syndetic. By Proposition \ref{Combinatorial}, this means that there is $n\geq m$ so that for every $x\in H_n$, we have $S\cap (x\circ \emb{\mathbf{A}_m, \mathbf{A}_n})\neq \emptyset$. Let $\emb{\mathbf{A}_m, \mathbf{A}_n} = \{f_1,...,f_k\}$, and find $g_1,...,g_k\in G_d$ with $\pi_m(g_i) = f_i$. Fix $g\in G_d$. Then for some $i\leq k$, we have $(gg_i)|_m\in S$. But this implies that $\bigcup_{i\leq k} \pi_m^{-1}(S)g_i^{-1} = G_d$, so $\pi_m^{-1}(T) = G_d\setminus \pi_m^{-1}(S)$ is not thick.
\qedhere
\vspace{3 mm}

\section{Piecewise syndetic sets}
In the previous section, we successfully proved analogues of Definition \ref{ComboDef} and Fact \ref{SemigroupDef} for thick subsets of $H_m$ using a definition resembling Fact \ref{ShiftDef}. It easily follows that these analogues also hold for syndetic subsets of $H_m$, in turn providing a generalization of Theorem 2.9(c) from [BHM]. However, we can not yet conclude that these analogues hold for piecewise syndetic sets. 

Let's briefly consider what this even means. As we have pointed out before, there is no right action on the set $H_m$, so it is not immediate what an analogue of Definition \ref{ComboDef} even looks like for piecewise syndetic sets. Theorem \ref{Sanity} suggests the following statement.

\begin{theorem}
\label{SanityPWS}
$P\subseteq H_m$ is piecewise syndetic iff $\pi_m^{-1}(P)\subseteq G_d$ is piecewise syndetic.
\end{theorem}

\proof
The surjection $\pi_m: G_d\rightarrow H_m$ induces a continuous embedding $\hat\pi_m: 2^{H_m}\rightarrow 2^{G_d}$, where given $S\subseteq H_m$, we have $\hat\pi_m(\chi_S) = \chi_{\pi_m^{-1}(S)}$. The group $G_d$ acts on the right on both spaces, and the map $\hat\pi_m$ is $G_d$-equivariant. The statements from Definition \ref{MainDef} do not depend on the topology of $G$, so also hold for $G_d$.

Suppose $P\subseteq H_m$ is piecewise syndetic, and let $S\subseteq H_m$ be syndetic with $\chi_S\in \overline{\chi_P\cdot G_d}$. Then $\chi_{\pi_m^{-1}(S)}\in \overline{\chi_{\pi_m^{-1}(P)}\cdot G_d}$, and $\pi_m^{-1}(S)\subseteq G_d$ is syndetic by Theorem \ref{Sanity}. It follows that $\pi_m^{-1}(P)$ is piecewise syndetic.

Conversely, suppose $P\subseteq H_m$ is a set with $\pi_m^{-1}(P)\subseteq G_d$ piecewise syndetic. Then there is syndetic $W\subseteq G_d$ with $W\in \overline{\chi_{\pi_m^{-1}(P)}\cdot G_d}$. But since $2^{H_m}$ is compact, we must have that $W = \pi_m^{-1}(S)$ for some $S\subseteq H_m$, and this $S$ is syndetic by Theorem \ref{Sanity}. It follows that $P$ is piecewise syndetic.
\qedhere
\vspace{3 mm}

Let us briefly return to the setting where $G$ is a discrete group. It is a fact that $P\subseteq G$ is piecewise syndetic iff $P = S\cap T$ for some syndetic $S\subseteq G$ and thick $T\subseteq G$. Let us recall one direction of the proof. Say $P\subseteq G$ is piecewise syndetic, and let $g_1,...,g_k\in G$ with $T:= Pg_1^{-1}\cupdots Pg_k^{-1}$ thick; we may assume that $k$ is minimal and that $g_1 = 1_G$ (if $T$ is thick, then so is $Tg_1$). Then setting $S = (G\setminus T)\cup P$, we have that $S$ is syndetic (since $Sg_1^{-1}\cupdots Sg_k^{-1} = G$) and that $P = S\cap T$. Notice that $T\setminus P\subseteq G\setminus S$ is not thick.

Returning to the setting where $G = \aut{\mathbf{K}}$, we can ask whether these statements also hold for $H_m$. Given a thick $T\subseteq H_m$, set $\mathrm{Dest}(T) = \{P\subseteq H_m: T\setminus P \text{ is not thick}\}$; if $P\in \mathrm{Dest}(T)$, we say that $P$ \emph{destroys} $T$. The main theorems of this section are the following, generalizing Theorem 2.4 from [BHM]:

\begin{theorem}
\label{Destruction}
$P\subseteq H_m$ is piecewise syndetic iff $P\in \mathrm{Dest}(T)$ for some thick $T\subseteq H_m$.
\end{theorem}

\begin{theorem}
\label{ThickIntSynd}
$P\subseteq H_m$ is piecewise syndetic iff there are thick $T\subseteq H_m$ and syndetic $S\subseteq H_m$ with $P = S\cap T$.
\end{theorem}

Let us remark that Theorems \ref{Destruction} and \ref{ThickIntSynd} are not corollaries of Theorem \ref{SanityPWS}. The best we can do is to note that given $S\subseteq H_m$ and $g\in G$, we have $\pi_m^{-1}(S)g^{-1} = \pi_n^{-1}(g|_m(S))$ for $n\geq m$ large enough with $g|_m(\mathbf{A}_m)\subseteq \mathbf{A}_n$. Then using Theorem \ref{SanityPWS}, we can say that if $P\subseteq H_m$ is piecewise syndetic, there is some $n\geq m$ and some thick $T\subseteq H_n$ with $i^n_m(P)\in \mathrm{Dest}(T)$. However, it is not immediately clear that we can take $n= m$. In the preprint [Z1], it is mistakenly asserted (see the discussion at the end of page 23) that Theorem \ref{Destruction} is somehow obvious. So while the result is true, the proof is a bit more difficult.

Let us start by obtaining a more combinatorial characterization of piecewise syndetic sets. The proof is similar to the proof of Proposition \ref{Combinatorial}.

\begin{prop}
\label{ComboPWS}
$P\subseteq H_m$ is piecewise syndetic iff there is $n\geq m$ and an exhausting sequence $\vec{X} = \{x_N: N<\omega\}$ so that for every $x\in \emb{\mathbf{A}_n, \vec{X}}$, we have \newline $P\cap (x\circ \emb{\mathbf{A}_m, \mathbf{A}_n})\neq \emptyset$.
\end{prop}

\proof
Suppose $P\subseteq H_m$ is piecewise syndetic. Find $g_N\in G$ a sequence of group elements so that $\chi_P\cdot g_N\rightarrow \chi_S$ for some syndetic $S\subseteq H_m$. We may assume that $\chi_P\cdot g_N(f) = \chi_S(f)$ for every $f\in \emb{\mathbf{A}_m, \mathbf{A}_N}$. Since $S$ is syndetic, there is some $n\geq m$ with $S\cap (x\circ \emb{\mathbf{A}_m, \mathbf{A}_n})\neq\emptyset$ for every $x\in H_n$. Setting $x_N = g_N|_N$, it follows that the conclusion holds.

Conversely, suppose $P\subseteq H_m$ is a set with the property described by the proposition. Find $g_N\in G$ with $g_N|_N = x_N$. Pass to a subsequence so that $\chi_P\cdot g_N|_N\rightarrow \chi_S$ for some $S\subseteq H_m$. This $S$ has the property that $S\cap x\circ \emb{\mathbf{A}_m, \mathbf{A}_n}\neq\emptyset$ for every $x\in H_n$, so $S$ is syndetic, and $P$ is therefore piecewise syndetic.
\qedhere
\vspace{3 mm}

Recall that $\mathcal{P}_m$ is the collection of piecewise syndetic subsets of $H_m$. Given a set $X$ and a family $\mathcal{S}$ on $X$, we say that $\mathcal{S}$ is \emph{partition regular} if given $S\in \mathcal{S}$, then whenever we write $S = S_1\cup S_2$, some $S_i\in \mathcal{S}$.

\begin{theorem}
\label{PWSPR}
$\mathcal{P}_m$ is partition regular.
\end{theorem}

\begin{proof}
Let $P = P_1\cup P_2$ be piecewise syndetic; using Proposition \ref{ComboPWS}, say this is witnessed by $n\in \mathbb{N}$ and $(x_N)_{N\geq n}$, $x_N\in H_N$. If $P_1$ is not piecewise syndetic, then for each $n_0\in \mathbb{N}$, there are $N\in \mathbb{N}$ and $y_{n_0}\in x_N\circ \emb{\mathbf{A}_{n_0}, \mathbf{A}_n}$ with $P_1\cap (y_{n_0}\circ \emb{\mathbf{A}_m, \mathbf{A}_{n_0}}) = \emptyset$. But now we see that $n\in \mathbb{N}$ and $(y_{n_0})_{n_0\geq n}$ witness that $P_2$ is piecewise syndetic.
\end{proof}

Write $\mathcal{N}_m := \mathcal{P}(H_m)\setminus \mathcal{P}_m$ for the collection of not-piecewise-syndetic subsets of $H_m$. We have just proven that $\mathcal{N}_m$ is an ideal. There is another natural ideal one could consider on $H_m$:
\begin{align*}
\mathcal{I}_m := \{A\subseteq H_m: \forall\, T\in \mathcal{T}_m (T\setminus A \in \mathcal{T}_m)\}.
\end{align*}
In other words, members of $\mathcal{I}_m$ are exactly those sets which do not destroy any thick set. We will eventually see that $\mathcal{N}_m = \mathcal{I}_m$, and we will use this to prove Theorems \ref{Destruction} and \ref{ThickIntSynd}.

\begin{prop}
\label{NSubsetI}
$\mathcal{N}_m\subseteq \mathcal{I}_m$.
\end{prop}

\begin{proof}
Let $W\in \mathcal{N}_m$, and fix $T\in \mathcal{T}_m$. Find for each $N\geq m$ an $x_N\in H_N$ so that $x_N\circ \emb{\mathbf{A}_m, \mathbf{A}_N}\subseteq T$. As $W$ is not piecewise syndetic, we can find for each $n\geq m$ an $N\geq n$ and a $y_n\in x_N\circ \emb{\mathbf{A}_n, \mathbf{A}_N}$ with $W\cap (y\circ \emb{\mathbf{A}_m, \mathbf{A}_n}) = \emptyset$. But now set $U = \bigcup_n y_n\circ \emb{\mathbf{A}_m, \mathbf{A}_N}$. We see that $U$ is thick and $U\subseteq T\setminus W$.
\end{proof}

To prove the other inclusion, we need a characterization of piecewise syndetic sets similar to Fact \ref{SemigroupDef}. Call $S\subseteq H_m$ \emph{minimal} if $\overline{\chi_S\cdot G}$ is a minimal $G$-flow. Note that any non-empty minimal set is syndetic.

\begin{prop}
\label{PWSM}
Let $P\subseteq H_m$. Then $P$ is piecewise syndetic iff there is some minimal right ideal $Y\subseteq \varprojlim \beta H_n$ and some $\alpha\in Y$ with $P\in \alpha(m)$. Equivalently, $P$ is piecewise syndetic iff $\tilde{P}\subseteq S(G)$ meets some minimal right ideal.
\end{prop}

\begin{proof}
First, suppose there are $Y$ and $\alpha\in Y$ as above with $P\in \alpha(m)$. Then $\overline{\chi_P\cdot \alpha\cdot G}$ is a minimal $G$-flow, so $\alpha^{-1}(P)$ is a minimal set. Since $i_m\in \alpha^{-1}(P)$, we see that $\alpha^{-1}(P)$ is syndetic, so $P$ is piecewise syndetic.

For the converse, suppose $P$ is piecewise syndetic, where $S$ is syndetic and $\chi_S = \chi_P\cdot \alpha$, $\alpha\in \varprojlim \beta H_n$. By replacing $\alpha$ with $\alpha\cdot \gamma$ for $\gamma\in \varprojlim \beta H_n$ in some minimal right ideal, we can assume that $\alpha$ is in some minimal right ideal $Y$. Fix $g\in G$ with $g|_m\in \alpha^{-1}(P)$. Then $P\in \alpha g(m)$ and $\alpha\cdot g\in Y$.
\end{proof}

\begin{lemma}
\label{ThickAvoidI}
Let $F_m\in \beta(\mathcal{T}_m)$. Then for any $A\in \mathcal{I}_m$, we have $H_m\setminus A\in F_m$. 
\end{lemma}

\begin{proof}
Fix $A\in \mathcal{I}_m$, and let $B = H_m\setminus A$. Then for every $T\in \mathcal{T}_m$, we have $T\cap B\in \mathcal{T}_m$. In particular, $F_m\cup\{B\}$ generates a thick filter, so by maximality of $F_m$ we have $B\in F_m$.
\end{proof}

\begin{theorem}
\label{NEqualsI}
$\mathcal{N}_m = \mathcal{I}_m$.
\end{theorem}

\begin{proof}
By Proposition \ref{NSubsetI}, it is enough to prove $\mathcal{I}_m\subseteq \mathcal{N}_m$. Suppose that $P\subseteq H_m$ is piecewise syndetic. We need to show that $P\not\in \mathcal{I}_m$. By Proposition \ref{PWSM}, fix a minimal right ideal $Y = \varprojlim Y_n\subseteq \varprojlim \beta H_n$ and an $\alpha\in Y$ with $P\in \alpha(m)$. Let $F_n$ be the filter of clopen neighborhoods of $Y_n$; by Theorem \ref{ThickUlts}, we have $F_n\in \beta(\mathcal{T}_n)$ for each $n\in \mathbb{N}$. But suppose $P$ were in $\mathcal{I}_m$; by Lemma \ref{ThickAvoidI}, we have $H_m\setminus P\in F_m$, so in particular $H_m\setminus P\in \alpha(m)$. This is a contradiction.
\end{proof}

Note that Theorem \ref{NEqualsI} implies Theorem \ref{Destruction}. We conclude this section with the (now short) proof of Theorem \ref{ThickIntSynd}.

\begin{proof}[Proof of Theorem \ref{ThickIntSynd}]\mbox{}\\*
Assume $P$ is piecewise syndetic. As $P\not\in \mathcal{I}_m$, find a thick $T\subseteq H_m$ with $T\setminus P$ not thick; we may assume $P\subseteq T$. So $H_m\setminus (T\setminus P) = P\cup (H_m\setminus T) := S$ is syndetic, and $P = T\cap S$.

For the other direction, say $T\in \mathcal{T}_m$ and $S\in \mathcal{S}_m$. Then $T\setminus T\cap S$ is not thick, so $T\cap S$ is piecewise syndetic.
\end{proof}

\begin{rem}
The proof of Theorem \ref{NEqualsI} gives us the following strengthening of Corollary \ref{ThickIntSynd}. If $P\subseteq H_m$ is piecewise syndetic and $P$ is $F$-positive for some $F\in \beta(\mathcal{T}_m)$, then there is $T\in F$ with $T\setminus P$ not thick. As we may assume $T\supseteq P$, we have that $H_m\setminus (T\setminus P) := S$ is syndetic, and $P = T\cap S$.
\end{rem}

\section{Questions}
In this final section, we collect some questions suggested by the previous results. The first question is an abstract one about families.

 \begin{que}
\label{WhenRegular}
If $(X, \mathcal{S})$ is a family and $Y$ is a set, then what are necessary and sufficient conditions for $\phi: (X,\mathcal{S})\rightarrow Y$ to be regular?
\end{que}
\vspace{2 mm}

The ``map'' language we have been using is somewhat misleading; whether or not a map $\phi: (X, \mathcal{S})\rightarrow Y$ is regular or strong depends only on the equivalence relation $E_\phi$, where $x E_\phi x'$ iff $\phi(x) = \phi(x')$. In particular, we can restrict our attention to the case where $\phi$ is surjective. 

When thinking about Question \ref{WhenRegular}, it might be useful to ``turn the problem around.'' Fix a surjection $\phi: X\rightarrow Y$. If $\mathcal{T}$ is a family on $Y$, set 
\begin{align*}
\phi^{-1}(\mathcal{T}) = \{\mathcal{S}: \mathcal{S} \text{ is a family on $X$ and } \phi(\mathcal{S}) = \mathcal{T}\}.
\end{align*}
Two families in $\phi^{-1}(\mathcal{T})$ are worth distinguishing. We set 
\begin{align*}
\phi^{-1}_{min}(\mathcal{T}) &= \{A\subseteq X: \exists B\in \mathcal{T} (\phi^{-1}(B)\subseteq A)\},\\
\phi^{-1}_{max}(\mathcal{T}) &= \{A\subseteq X: \phi(A)\in \mathcal{T}\}.
\end{align*}
It is not hard to check that for any family $\mathcal{S}\in \phi^{-1}(\mathcal{T})$, we have $\phi^{-1}_{min}(\mathcal{T})\subseteq \mathcal{S}\subseteq \phi^{-1}_{max}(\mathcal{T})$. Also note that $\phi: (X,\mathcal{S})\rightarrow Y$ is strong iff $\mathcal{S} = \phi^{-1}_{max}(\mathcal{T})$. When $\mathcal{T}$ and $\phi$ are understood and $\mathcal{S}\in \phi^{-1}(\mathcal{T})$, call $\mathcal{S}$ \emph{regular} if the map $\phi: (X, \mathcal{S})\rightarrow Y$ is regular.

\begin{prop}
\label{MinRegular}
$\mathcal{S} = \phi^{-1}_{min}(\mathcal{T})$ is regular. 
\end{prop}

\proof
Let $p\in \beta(\mathcal{S})$. For $A\subseteq X$, let $\tilde{A}$ denote the largest $E_\phi$-invariant subset of $A$. By our assumption that $\mathcal{S} = \phi^{-1}_{min}(\mathcal{T})$, we have that if $A\in p$, then also $\tilde{A}\in p$.

Suppose towards a contradiction that $\phi(p)\subsetneq q$ for some $q\in \beta(\mathcal{T})$. Fix $B\in q\setminus \phi(p)$. Then $\phi^{-1}(B)\not\in p$, so find $A\in p$ with $\phi^{-1}(B)\cap A\not\in \mathcal{S}$. We may assume $A = \tilde{A}$. But now $B\cap \phi(A)\not\in \mathcal{T}$, a contradiction since we have $B$ and $\phi(A)$ in $q$.
\qedhere
\vspace{3 mm}

Given $\mathcal{S}, \mathcal{S}'\in \phi^{-1}(\mathcal{T})$, say that $\mathcal{S}'$ is a \emph{conservative over} $\mathcal{S}$ if $\mathcal{S}\subseteq \mathcal{S}'$ and for every $p\in \beta(\mathcal{S}')$, we have $p\cap \mathcal{S}\in \beta(\mathcal{S})$. If $\mathcal{S}$ is regular and $\mathcal{S}'$ is conservative over $\mathcal{S}$, then $\mathcal{S}'$ is also regular.

\begin{prop}
\label{Conservative}
With $\mathcal{T}$ and $\phi$ as above, $\mathcal{S}\in \phi^{-1}(\mathcal{T})$ is regular iff $\mathcal{S}$ is conservative over $\phi^{-1}_{min}(\mathcal{T})$.
\end{prop}

\proof
One direction is clear. For the other direction, suppose $\mathcal{S}$ is regular, and fix $p\in \beta(\mathcal{S})$. Towards a contradiction, suppose $p\cap \phi^{-1}_{min}(\mathcal{T})\subsetneq q$ with $q\in \beta(\phi^{-1}_{min}(\mathcal{T}))$, and let $A\in q\setminus p$. We may assume $A = \tilde{A}$, i.e.\ that $A$ is $E_\phi$-invariant. But we have $\phi(A)\in \phi(q) = \phi(p)$, so $\phi^{-1}(\phi(A)) = A\in p$, a contradiction.
\qedhere
\vspace{5 mm}

The next question concerns both abstract families and the families $\mathcal{T}_m$ of thick sets on $H_m$. Suppose $X$ is a set and $F$, $F'$ are filters on $X$. In a slight abuse of language, call $F$ and $F'$ \emph{disjoint} if there are $A\in F$ and $B\in F'$ with $A\cap B = \emptyset$. If $\mathcal{S}$ is a family on $X$, we say that $\mathcal{S}$ has \emph{disjointness} if whenever $p\neq q\in \beta(\mathcal{S})$, we have $p$ and $q$ disjoint.

\begin{que}
\label{Disjointness}
Characterize those families $\mathcal{S}$ with disjointness.
\end{que}

Suppose briefly that $G$ is a discrete group. If $\mathcal{T}$ is the family of thick subsets of $G$, then $\mathcal{T}$ has disjointness. This is because the members of $\beta(\mathcal{T})$ are exactly the filters of clopen neighborhoods of minimal right ideals of $\beta G$, and minimal right ideals of $\beta G$ are either identical or disjoint.

Returning to the case that $G = \aut{\mathbf{K}}$, while it is still true that distinct minimal right ideals of $S(G)$ are disjoint, this is not necessarily true on each level. So for the family $\mathcal{T}_m$, Question \ref{Disjointness} is equivalent to the following:

\begin{que}
\label{ThickDisjoint}
Let $Y = \varprojlim Y_n$ and $Y' = \varprojlim Y_n'$ be distinct minimal right ideals of $S(G) = \varprojlim \beta H_n$. If $Y_m\neq Y_m'$, must we have $Y_m\cap Y_m' = \emptyset$?
\end{que}
\vspace{3 mm}

The final question is specific to the families $\mathcal{T}_m$, or at least to reasonably ``definable'' families. Theorem \ref{Destruction} guarantees that to every piecewise syndetic $P\subseteq H_m$, there is a thick $T\subseteq H_m$ so that $P$ destroys $T$. This question asks whether there is a ''definable'' choice of $T$.

\begin{que}
\label{DefinableDest}
Let $X = \{(P, T)\in 2^{H_m}\times 2^{H_m}: T \text{ is thick and } P\in \mathrm{Dest}(T)\}$. Does $X$ admit a Borel uniformization? In other words, is there a Borel partial function $\psi: 2^{H_m}\rightarrow 2^{H_m}$ with domain the piecewise syndetic sets so that $\psi(P)$ is thick and $P$ destroys $\psi(P)$?
\end{que}

In the setting where $G$ is a \emph{countable} discrete group, the answer to the analogous question is positive. If $P$ is piecewise syndetic, search for a finite $E = \{g_1,...,g_k\}\subseteq G$ with $T:= Pg_1\cupdots Pg_k$ thick so that $k$ is minimal. Then set $\psi(P) = Tg_1^{-1}$.

Andy Zucker

Carnegie Mellon University

Pittsburgh, PA 15213

andrewz@andrew.cmu.edu

\end{document}